\documentclass{amsart}

\usepackage{mathrsfs}
\usepackage[colorlinks=true]{hyperref}

\newtheorem{theorem}{Theorem}[section]
\newtheorem{lemma}[theorem]{Lemma}
\newtheorem{corollary}[theorem]{Corollary}
\newtheorem{proposition}[theorem]{Proposition}

\theoremstyle{definition}

\newtheorem{conjecture}[theorem]{Conjecture}

\theoremstyle{remark}
\newtheorem{remark}[theorem]{Remark}

\numberwithin{equation}{section}

\def\scal#1#2{\langle #1, #2\rangle}
\def\R#1{\mathbb{R}^{#1}}

\DeclareMathOperator{\trace}{\mathrm{tr}}

\def\F{\mathbb{F}}

\def\half#1#2{\begin{matrix}\frac{#1}{#2}\end{matrix}}

\begin{document}

\title{On the non-vanishing property for real analytic solutions of the $p$-Laplace equation}

\author{Vladimir G. Tkachev}
\address{Department of Mathematics, Link\"oping University, Sweden}
\email{vladimir.tkatjev@liu.se}
\thanks{}


\date{March 10, 2015}



\begin{abstract}
By using a nonassociative algebra argument, we prove that $u\equiv0$ is the only cubic homogeneous polynomial solution to the $p$-Laplace equation $\mathrm{div} |Du|^{p-2}Du(x)=0 $ in $\R{n}$ for any $n\ge2$ and $p\not\in\{0,2\}$.
\end{abstract}

\maketitle

\section{Introduction}
In this paper, we continue to study applications of nonassociative algebras to elliptic PDEs started in \cite{Tk14}, \cite{NTVbook}.
Let us consider the $p$-Laplace equation
\begin{equation}\label{plaplace}
\Delta_p u:=|D u|^2\Delta u+\half{p-2}{2}\scal{Du}{D|Du|^2}=0.
\end{equation}
Here $u(x)$ is a function defined on a domain $E\subset \R{n}$, $Du$ is its gradient and $\scal{}{}$ denotes the standard inner product in $\R{n}$. It is well-known that for $p>1$ and $p\ne2$ a weak (in the distributional sense) solution to (\ref{plaplace}) is normally in the class $C^{1,\alpha}(E)$ \cite{Ural68}, \cite{Uhlenbeck}, \cite{Evans82}, but need not to be a H\"older continuous or even continuous in a closed domain with nonregular boundary \cite{KrolMaz72}. On the other hand, if $u(x)$ is a weak solution of (\ref{plaplace}) such that  $\mathrm{ess} \sup |Du(x)|>0$ holds locally in a domain $E\subset \R{n}$ then $u(x)$ is in fact a real analytic function in $E$ \cite{Lewis77}. 

An interesting problem is whether the  converse non-vanishing property holds true. More precisely: is it true that any real analytic solution $u(x)$ to (\ref{plaplace}) for $p>1$, $p\ne2,$  in a domain $E\subset \R{n}$ with vanishing gradient $Du(x_0)=0$ at some $x_0\in E$ must be identically zero? Notice that the analyticity assumption is necessarily because for any $d\ge 2$ and $n\ge2$ there exists plenty non-analytic $C^{d,\alpha}$-solutions $u(x)\not\equiv 0$ to (\ref{plaplace}) in $\R{n}$ for which $Du(x_0)=0$ for some $x_0\in \R{n}$, see \cite{Krol73}, \cite{Aronsson86},  \cite{KSVeron}, \cite{Veron}, \cite{Tk06c}. 

The non-vanishing property was first considered and solved in affirmative in $\R{2}$ by John L. Lewis in  \cite{Lewis80} as a corollary of the following crucial result (Lemma~2 in \cite{Lewis80}): if $u(x)$ is a real  homogeneous polynomial of degree $m=\deg u\ge 2$ in $\R{2}$ and $\Delta_pu(x)=0$  for $p>1$, $p\ne2$ then $u(x)\equiv 0$. Concerning the general case $n\ge 3$, it is not difficult to see (see also Remark~4 in \cite{Lewis80})  that the non-vanishing property for real analytic solutions to (\ref{plaplace}) in $\R{n}$ is equivalent to following conjecture.

\begin{conjecture}
\label{con:1}
Let $u(x)$ be a real  homogeneous polynomial of degree $m=\deg u\ge 2$ in $\R{n}$, $n\ge3$. If $\Delta_pu(x)=0$  for $p>1$, $p\ne2$ then $u(x)\equiv 0$.
\end{conjecture}

Notice that a simple analysis shows that Conjecture~\ref{con:1} is true for $m=2$ and any dimension $n\ge2$, therefore the only interesting case is when $m\ge3$. In \cite{Lewis80}, Lewis mentioned that Conjecture~\ref{con:1} holds also true for $n=m=3$ (unpublished). In 2011, J.L.~Lewis asked the author whether Conjecture~\ref{con:1} remains true  for any $n\ge3$ and $m\ge3$. In this paper we obtain the following partial result for the cubic polynomial case.

\begin{theorem}\label{th:main}
Conjecture~\ref{con:1} is true for $m=3$ and any $n\ge2$. More precisely, if $u(x)$ is a homogeneous degree three solution of (\ref{plaplace}) in $\R{n}$, $n\ge2$ and $p\not\in\{0,2\}$ then $u(x)\equiv 0$.
\end{theorem}

It follows from the above discussion that the following property holds true.

\begin{corollary}
Let $u(x)\not\equiv0$ be a real analytic solution of (\ref{plaplace}) in a domain $E\subset\R{n}$, $n\ge2$ and $p\not\in\{0,2\}$. If $Du(x_0)=0$ at some point $x_0\in E$ then $D^4u(x_0)=0$.
\end{corollary}

\begin{remark}
Concerning Theorem~\ref{th:main}, notice that for $p=2$, there is a reach class of homogeneous polynomial solutions of (\ref{plaplace}) of any degree $m\ge1$. In  other exceptional case, $p=0$ one easily sees that $u(x)=(a_1x_1+\ldots+a_nx_n)^3$ is  a cubic polynomial solution to (\ref{plaplace})  in any dimension $n\ge1$.
\end{remark}

\begin{remark}In the limit case $p=\infty$, an elementary argument (see Proposition~\ref{pro:infty} below) yields the non-vanishing property for real analytic solutions of the $\infty$-Laplacian \begin{equation}\label{infplaplace}
\Delta_\infty u:=\scal{Du}{D|Du|^2}=0.
\end{equation}
On the other hand, it is interesting to note that, in contrast to the case $p\ne \infty$, the non-vanishing property holds still true for H\"older continuous  $\infty$-harmonic functions. Namely,  for $C^2$-solutions of (\ref{infplaplace}) and $n=2$ the non-vanishing property was established by G.~Aronnson \cite{Aronsson68}. In any dimension $n\ge2$ it was proved for $C^4$-solutions by L.~Evans \cite{Evans93} and for $C^2$-solutions by Yifeng Yu \cite{YuY}. The non-vanishing property  for $C^2$-smooth $\infty$-harmonic maps was recently established by N.~Katzourakis \cite{Kat14}.
\end{remark}

The proof of Theorem~\ref{th:main} is by contradiction and makes use a nonassociative algebra argument which was earlier applied for an eiconal type equation in \cite{Tk10a}, \cite{Tk14} and study of Hsiang cubic minimal cones \cite{NTVbook}. First, in section~\ref{sec:1} we recall the definition of a metrised algebra and give some preparatory results. In particular,  in Proposition~\ref{pro:nonass} we reformulate the original PDE-problem for cubic polynomial solutions as the existence of a metrised non-associative algebra structure on $\R{n}$ satisfying a certain fourth-order identity. Then  in Proposition~\ref{pr:equiv}, we show that any such algebra must be zero, thus implying the claim of Theorem~\ref{th:main}.

\section{Preliminaries}\label{sec:1}

\subsection{Metrised algebras}
By an algebra on a vector space $V$ over a field $\F$ we mean an $\F$-bilinear form  $(x,y)\to xy\in V$, $x,y\in V$, also called the multiplication and in what follows denoted by juxtaposition. An algebra $V$ is called a zero algebra if $xy=0$ for all $x,y\in V$.

Suppose that $(V,Q)$ is an inner product vector space, i.e. a vector space $V$ over a field $\F$ with a non-degenerate bilinear symmetric form $Q:V\otimes V\to \F{}$. The inner product $Q$ on an algebra $V$ is called {associative} (or invariant) \cite{Bordemann}, \cite[p.~453]{Knus} if
\begin{equation}\label{Qass}
Q(xy,z)=Q(x,yz), \qquad \forall x,y,z\in V.
\end{equation}
An algebra $V$ with an associative inner product is called \textit{metrised} \cite{Bordemann}, \cite[Ch.~6]{NTVbook}.

In what follows, we assume that $\F=\R{}$ and that $(V,Q)$ is a commutative, but may be non-associative metrised algebra. Let us consider the cubic form
$$
u(x):=Q(x^2,x):V\to \R{}.
$$
Then it  is easily verified that the multiplication $(x,y)\to xy$ is uniquely determined by the identity
\begin{equation}\label{Qmul}
Q(xy,z)=u(x;y;z),
\end{equation}
where
$$
u(x;y;z):=u(x+y+z)-u(x+y)-u(x+z)-u(y+z)+u(x)+u(y)+u(z)
$$
is a symmetric trilinear form obtained by the linearization of $u$. For further use notice the following corollary of the homogeneity of $u(x)$:
\begin{equation}\label{hom}
u(x;x;y)=2\partial_y u|_{x}.
\end{equation}

In the converse direction, given a cubic form $u(x):V\to \R{}$ on an inner product vector space $(V,Q)$, (\ref{Qmul}) yields a non-associative commutative algebra structure on $V$ called the Freudenthal-Springer algebra of the cubic form $u(x)$ and denoted by $V^{\mathrm{FS}}(Q,u)$, see for instance \cite[Ch.~6]{NTVbook}). According to the definition, $V^{\mathrm{FS}}(Q,u)$ is a metrised algebra with an associative inner product $Q$.

We point out that the multiplication operator $L_x:V\to V$ defined by $L_xy=xy$  is self-adjoint with respect to the inner product $\scal{}{}$. Indeed, it follows from the symmetricity of $u(x,y,z)$ that
$$
Q(L_xy,z)=Q(xy,z)=Q(y,xz)=Q(y,L_xz).
$$

Furthermore, for $k\ge1$ one defines the $k$th principal power of $x\in V$ by
\begin{equation}\label{xk}
x^k=L_x^{k-1}x=\underbrace{x(x(\cdots (xx)\cdots))}_{k \text{ copies of }x}
\end{equation}
In particular, we write $x^2=xx$ and $x^3=xx^2$. Since $V$ is non-associative, in general $x^kx^m\ne x^{k+m}$. However, one easily verifies that the latter power-associativity holds for $k+m\le3$.

\smallskip
We recall that an element $c\in V$ is called an \textit{idempotent} if $c^2=c$. By $\mathscr{I}(V)$ we denote the set of all non-zero idempotents of $V$.

\begin{lemma}\label{Lem:1}
Let $(V,Q)$ be a non-zero commutative metrised algebra with positive definite inner product $Q$. Then $\mathscr{I}(V)\ne \emptyset$.
\end{lemma}

\begin{proof}
First notice that the cubic form $u(x):=Q(x^2,x)\not\equiv 0$, because otherwise the linearization would yield $Q(xy,z)\equiv 0$ for all $x,y,z\in V$, implying $xy\equiv 0$, i.e. $V$ is a zero algebra, a contradiction. Next notice that in virtue of the positive definiteness assumption, the unit sphere $S=\{x\in V:Q(x)=1\}$ is  compact in the standard Euclidean topology on $V$. Therefore as $u$ is a continuous function on $S$, it attains its maximum value at some point $y\in S$, $Q(y)=1$. Since $u\not\equiv 0$ is an odd function, the maximum value $u(y)$ must be strictly positive and the stationary equation $\partial_x u|_{y}=0$ holds whenever $x\in V$ satisfies the tangential condition
\begin{equation}\label{tang}
Q(y;x)=0.
\end{equation}
Using (\ref{hom}) and (\ref{Qmul}) we have
$$
0=\partial_x u|_{y}=\frac{1}{2}u(y;y;x)=\half{1}{2}Q(y^2;x)
$$
which implies in virtue of the non-degeneracy of $Q$ and (\ref{tang}) that $y^2=ky$, for some $k\in \R{\times}$. It follows that
$$
kQ(y;y)=Q(y^2;y)=u(y)>0,
$$
which yields $k\ne0$. Then  setting $c=y/k$ we obtain $c^2=c$, i.e. $c\in \mathscr{I}(V)$.
\end{proof}

\begin{remark}
In a general finite-dimensional non-associative algebra over $\R{}$, there exist either an idempotent or an absolute nilpotent, see a topological proof, for example, in \cite{Lyubich1}.
\end{remark}

\subsection{Preliminary reductions}
Now suppose that $V=\R{n}$ be the  Euclidean space endowed with the standard inner product $Q(x;y)=\scal{x}{y}$. Let $u:V\to \R{}$ be a cubic homogeneous polynomial solution  of (\ref{plaplace}) and let $V^{\mathrm{FS}}(u)$ denotes  the corresponding Freudenthal-Springer algebra with multiplication $xy$ uniquely defined by
\begin{equation}\label{qscal}
\scal{xy}{z}=u(x;y;z).
\end{equation}
Then the homogeneity of $u(x)$ and (\ref{hom}) yield
\begin{equation}\label{hesu1}
\scal{x^2}{x}=u(x;x;x)=2\partial_x u|_{x}=6u(x).
\end{equation}
Similarly, it follows from (\ref{hom}) that
\begin{equation}\label{hesu3}
\scal{x^2}{y}=u(x;x;y)=2\partial_y u|_x=2\scal{Du(x)}{y}
\end{equation}
which yields the  expression for the gradient of $u$ as an element of the Freudenthal-Springer algebra:
\begin{equation}\label{L50}
Du(x)=\half{1}{2}x^2.
\end{equation}
A further polarization of (\ref{hesu3}) yields
$$
\scal{y}{D^2u(x)\, z}=u(x;y;z)=\scal{y}{L_xz},
$$
where $L_xy=xy$ is the multiplication operator by $x$ and $D^2u(x)$ is the Hessian operator of $u$. This implies
\begin{equation}\label{Lxx}
D^2u(x)=L_x,
\end{equation}

\begin{proposition}
\label{pro:nonass}
A cubic form $u:V=\R{n}\to \R{}$ satisfies $(\ref{plaplace})$  if and only if its Freudenthal-Springer algebra $V^{\mathrm{FS}}(u)$ satisfies the following identity:
\begin{equation}\label{L6}
\scal{b}{x}\scal{x^2}{x^2}+\half{p-2}{2}\scal{x^2}{x^3}=0
\end{equation}
where
\begin{equation}\label{L5}
b=b(V):=\sum_{i=1}^ne_i^2,
\end{equation}
and $e_1,\ldots,e_n$ is an arbitrary orthonormal basis of $\R{n}$.
\end{proposition}

\begin{proof}
Using (\ref{L50}) and (\ref{Lxx}), one  obtains
$$
\scal{Du}{D^2u|_{x}Du}=\half{1}{4}\scal{L_xx^2}{x^2} =\half{1}{4}\scal{x^3}{x^2},
$$
and similarly,
\begin{equation}\label{L11}
\begin{split}
\Delta u(x)&=\trace D^2u|_x=\trace L_x=\sum_{i=1}^n \scal{L_xe_i}{e_i}=\sum_{i=1}^n \scal{e_i^2}{x}=\scal{b(V)}{x},
\end{split}
\end{equation}
where $b$ is defined by (\ref{L5}). Inserting the found relations into  (\ref{plaplace}) yields (\ref{L6}). In the converse direction, if $V$ is a metrised algebra satisfying (\ref{L6}) then  $u(x)$ defined by  (\ref{hesu1}) is easily seen to satisfy (\ref{plaplace}).
\end{proof}

\section{Proof of Theorem~\ref{th:main}}
Using the introduced above definitions and Proposition~\ref{pro:nonass}, one easily sees that the following property is equivalent to Theorem~\ref{th:main}.

\begin{proposition}\label{pr:equiv}
A commutative metrised algebra $(V,Q)$ with $\dim V\ge2$ and satisfying (\ref{L11}) with $p\not\in\{0,2\}$, is a zero algebra.
\end{proposition}

\begin{proof}
We argue by contradiction and assume that $(V,\scal{}{})$ is a non-zero commutative metrised algebra satisfying (\ref{L6}).  Since $p\ne 2$, this identity  is equivalent to
\begin{equation}\label{L01}
\scal{q}{x}\scal{x^2}{x^2}+\scal{x^2}{x^3}=0,
\end{equation}
where
\begin{equation}\label{q}
q=\frac{2}{p-2}b(V)\in V.
\end{equation}
Polarizing (\ref{L01}) we obtain in virtue of
$$
\partial_yx^3=\partial_y(x(xx))=yx^2+2x(xy)
$$
and the associativity of the inner product that
$$
\scal{q}{y}\scal{x^2}{x^2}+
4\scal{q}{x}\scal{xy}{x^2}+
4\scal{xy}{x^3}+\scal{x^2}{yx^2}=0,
$$
implying by the arbitrariness of $y$ that
\begin{equation}\label{L02}
\scal{x^2}{x^2}q+4\scal{q}{x}x^3+4x^4+x^2x^2=0,
\end{equation}
we according to (\ref{xk}) $x^4=xx^3$.
A further polarization of (\ref{L02}) yields
$$
4\scal{x^2}{xy}q+4\scal{q}{y}x^3+4\scal{q}{x}(yx^2+2x(xy)) +4yx^3+4x(yx^2+2x(xy)) +4x^2(xy)=0,
$$
which implies an operator identity
\begin{equation}\label{L03}
2L_x^3+L_{x^3}+\scal{q}{x}(L_{x^2}+2L_x^2)+L_xL_{x^2}+L_{x^2}L_x+(q\otimes x^3+x^3\otimes q)=0.
\end{equation}
Here $a\otimes b$  denotes the rank one operator acting by $(a\otimes b)y=a\scal{b}{y}$.

Now, notice that by our assumption and Lemma~\ref{Lem:1}, $\mathscr{I}(V)\ne \emptyset$.  Let $c\in \mathscr{I}(V)$ be an arbitrary idempotent. Then setting $x=c$ in (\ref{L01}) we find
$$
|c|^2q+(4\scal{q}{c}+5)c=0.
$$
Taking scalar product of the latter identity with $c$ yields
\begin{equation}\label{C03}
\scal{q}{c}=-1,\qquad q=-\frac1{|c|^2}c
\end{equation}
in particular $q\ne0.$ Furthermore, setting $x=c$ in (\ref{L03}) and applying (\ref{C03}) yields
$$
2L_c^3+L_{c}+\scal{q}{c}(L_{c}+2L_c)+2L_c^2+(q\otimes c+c\otimes q)=2L_c^3-\frac{2}{|c|^2}c\otimes c=0,
$$
therefore
\begin{equation}\label{C04}
L_c^3=\frac{1}{|c|^2}c\otimes c.
\end{equation}
The latter identity, in particular, implies that
\begin{equation}\label{zero}
L_c=0 \text{ on } c^\bot:=\{x\in V: \scal{c}{x}=0\},
\end{equation}
where by the assumption $\dim c^\bot=\dim V-1\ge1$.

We claim that $c^\bot$ is a zero subalgebra of $V$. Indeed, if $x,y\in c^\bot$ then by the associativity of the inner product and (\ref{zero}) we have
$$
\scal{xy}{c}=\scal{x}{cy}=\scal{x}{L_cy}=0,
$$
hence $xy\in c^\bot$ which implies that $c^\bot$ is a subalgebra (in fact, an ideal) of $V$. Suppose that $c^\bot$ is a non-zero subalgebra, then it follows by Lemma~\ref{Lem:1} that there is a nontrivial idempotent in $c^\bot$, say $w$. Then by the  second identity in (\ref{C03}) we have $\scal{w}{q}=0$, therefore  (\ref{L01}) yields
\begin{equation}\label{C05}
\scal{w^2}{w^3}=|w|^2=0.
\end{equation}
The obtained contradiction proves our claim.

To finish the proof, we consider an arbitrary orthonormal basis $\{e_i\}_{1\le i\le n}$ of $V$ with $e_n=c/|c|$. Then $e_i\in c^\bot$ for all $1\le i\le n-1$, hence by the above zero-algebra property  we have $e_i^2=0$. Applying (\ref{L5}) we get
$$
\frac{p-2}{2}q=b(V)=\sum_{i=1}^ne_i^2=\frac{c}{|c|^2}=-q,
$$
which yields in virtue of $q\ne 0$ that $p=0$, a contradiction. The theorem is proved.
\end{proof}


\section{Concluding remarks}

We notice that the appearance of non-associative algebras in the above analysis of the $p$-Laplace equation is not accident and becomes more substantial if one considers the  following eigenfunction problem
\begin{equation}\label{radial}
\Delta_p u(x)=\lambda |x|^2u(x),\quad \lambda\in \R{},\quad p\ne2,
\end{equation}
with $u(x)$ being a cubic homogeneous polynomial. Notice that (\ref{plaplace}) correspond to $\lambda=0$ in (\ref{radial}). The problem (\ref{radial}) for $p=1$  has first appeared in Hsiang's study of cubic minimal cones in $\R{n}$ \cite{Hsiang67}. In fact, it follows from recent results in \cite[Ch.~6]{NTVbook} that any cubic polynomial solution of (\ref{radial}) is necessarily harmonic, and thus satisfies (\ref{radial}) for \textit{any} $p\ne2$! The zero-locus of any such solution is an algebraic minimal cone in $\R{n}$ \cite{Hsiang67}. Furthermore, it was shown  in \cite{NTVbook} that (\ref{radial}) has a large class of non-trivial cubic solutions for $p=1$ (and thus for any $p\ne 2$)  sporadically distributed over dimensions $n\ge2$. It turns out that these solutions have a deep relation to rank 3 formally real Jordan algebras and their classification requires a mush more delicate analysis by using nonassociative algebras, we refer to  \cite{Tk15} for more examples of solutions to (\ref{radial}) and their classification.

Finally, below we give an elementary  proof of the non-vanishing property for real analytic $\infty$-harmonic functions.

\begin{proposition}\label{pro:infty}
If $v(x)$ is a real analytic solution of the (\ref{infplaplace}) in a domain $D\subset \R{n}$ and $Dv(x_0)=0$ for some $x_0\in D\subset \R{n}$ then $v(x)\equiv v(x_0)$.

\end{proposition}

\begin{proof}
Indeed, we may assume that $x_0=0$ and suppose by contradiction that $v(x)\not\equiv v(0)$. Then a direct generalization of Lewis' argument given in Lemma~1 in \cite{Lewis80} easily yields the existence of a real homogeneous polynomial $u(x)\not\equiv 0$ of order $\deg u=k\ge2$ which also is a solution to (\ref{infplaplace}). Notice that $u(x)$ attains its maximum value on the unit sphere $S=\{x\in \R{n}:|x|=1\}$ at some point $y$. The stationary equation yields $Du(y)=\lambda y$ for some real $\lambda$ and by Euler's homogeneous function theorem
$$
ku(y)=\scal{y}{Du(y)}=\lambda |y|^2=\lambda
$$
and
$$
\scal{Du(y)}{D|Du|^2(y)}=\lambda (2k-2)|Du|^2(y)=2(k-1)\lambda^3,
$$
which yields by (\ref{infplaplace}) that $u(y)=0$, hence
$$
\max_{x\in S} u(x)=\frac{\lambda}{k}=0.
$$
A similar argument applied to the minimum value implies $\min_{x\in S} u(x)=0$, a contradiction with $u\not\equiv 0$ follows.
\end{proof}

\section*{Acknowledgements}

I would like to thank John L. Lewis for helpful discussion  and bringing my attention to Conjecture~\ref{con:1}.

\bibliographystyle{amsplain}

\def\cprime{$'$}
\providecommand{\bysame}{\leavevmode\hbox to3em{\hrulefill}\thinspace}
\providecommand{\MR}{\relax\ifhmode\unskip\space\fi MR }
\providecommand{\MRhref}[2]{%
  \href{http://www.ams.org/mathscinet-getitem?mr=#1}{#2}
}
\providecommand{\href}[2]{#2}

\end{document}